\documentclass[11pt,a4paper,reqno]{amsart}

\usepackage{amsmath,amsthm,amssymb,a4wide}
\usepackage{graphicx}
\usepackage{enumerate}

\def\K{\mathcal{K}}
\def\J{\mathcal{J}}
\def\L{\mathcal{L}}

\def\M{\mathcal{M}}
\def\S{\mathcal{S}}
\def\W{\mathcal{W}}

\def\RR{\mathbb{R}}

\def\d{\,{\rm d}}
 
\def\D{{\R \times \V}}

\def\R{\mathcal{R}}
\def\dR{{\partial \R}}
\def\V{\mathcal{V}}

\def\v{{\bf v}}
\def\s{\hat{\v}}
\def\r{{\bf r}}
\def\n{{\bf n}}

\def\tau{\ell}

\newtheorem{theorem}{Theorem}[section]

\newtheorem{lemma}[theorem]{Lemma}

\begin{document}
\title{An $L^p$ theory for stationary radiative transfer}
\date{\today}
\author{Herbert Egger \and Matthias Schlottbom}
\begin{abstract}
  We present a self-contained analysis of the stationary radiative transfer equation in weighted $L^p$ spaces. 
  The use of weighted spaces allows us to derive uniform a-priori estimates for $1 \le p \le \infty$ under minimal assumptions on the parameters. By constructing an explicit example, we show that our estimates are sharp and cannot be improved in general. Better estimates are however derived under additional assumptions on the parameters. We also present estimates for derivatives and traces of the solution and formulate a natural energy space, for which the data-to-solution map becomes an isomorphism. As a side result, we are able to prove uniform convergence of the source iteration for all $1 \le p \le \infty$ without the assumption of positive absorption that is frequently used in the literature.
\end{abstract}
\maketitle

\footnotetext{Department of Mathematics, Numerical Analysis and Scientific Computing, Technische Universit\"at Darmstadt, Dolivostr. 15, D--64293 Darmstadt, Germany.}

\small{
}

\section{Introduction}

The equilibrium distribution of mutually non interacting particles propagating through a scattering medium is described by the stationary radiative transfer equation \cite{CaseZweifel67,Cercignani:1988}
\begin{align}
\s \cdot \nabla \phi(\r,\v) + \sigma(\r,\v) \phi(\r,\v) 
&= \int_\V k(\r,\v', \v) \phi(\r,\v') \d \v' + f(\r,\v) \label{eq:rte_eq} \\
 \phi(\r,\v)&=g(\r,\v) \qquad\text{where } \n(\r) \cdot \v < 0. \label{eq:rte_bc}
\end{align}
This system arises as a basic model, e.g., in radiation hydrodynamics \cite{Pomraning73}, in reactor physics \cite{CaseZweifel67}, in astrophysics \cite{Peraiah04}, in climatology \cite{DM79}, or in optical tomography \cite{Arridge99}. Depending on the application, the function $\phi$, is called angular flux or specific intensity. The unit vector $\s = \v / |\v|$ denotes the direction and $\v \in \V$ is the velocity of propagation. The $\nabla$ operator only involves derivatives with respect to the spatial variable $\r \in \R$. 
Particles enter the system through interior sources $f$ or via a flux $g$ over the inflow boundary $\Gamma_-$, where
$$
  \Gamma_\pm:=\{(\r,\v)\in\partial\R\times \V:  \pm \n(\r)\cdot \v  > 0\}
$$
and $\n(\r)$ denotes the unit outward normal for $\r\in\dR$.
For physical reasons, the total cross-section $\sigma$ and the scattering kernel $k$ are non-negative 
functions of their arguments. Let us further define the scattering cross-sections
$$
\sigma_s(\r,\v) = \int_\V k(\r,\v,\v') \d\v'
\qquad \text{and} \qquad 
\sigma_s'(\r,\v) = \int_\V k(\r,\v',\v) \d \v',
$$
which describe scattering from or into direction $\v$. 
The difference $\sigma-\sigma_s$ is the absorption parameter.
In the following we shortly review some of the basic solvability results for \eqref{eq:rte_eq}--\eqref{eq:rte_bc} and recall the conditions they are based on:
Solvability in $L^1$ and $L^\infty$ has been established in \cite{CaseZweifel63}
under the sub-criticality conditions
$$
c(\r,\v) = \frac{\sigma_s(\r,\v)}{\sigma(\r,\v)}\le 1-\nu, \qquad c'(\r,\v) = \frac{\sigma_s'(\r,\v)}{\sigma(\r,\v)}  \le 1-\nu, \qquad \nu>0.
$$
These imply that the scattering operator is a small perturbation of the differential operator on the left hand side of \eqref{eq:rte_eq} and  contraction arguments apply. Corresponding results in $L^p$ for $1 \le p \le \infty$ can be found in \cite{Agoshkov98,ChoulliStefanov98}. 
A rather complete $L^p$ theory has been developed under similar conditions in \cite{DL93vol6} 
in the framework of semi-group theory.
Note that a-priori estimates for the solution derived under these conditions typically degenerate when $\nu \to 0$; see also Theorem~\ref{thm:apriori} below.
In \cite{Vladimirov61,Olhoeft62}, solvability in $L^1$ was established provided that 
$$
\sigma-\sigma_s \ge 0, \quad \sigma - \sigma_s' \ge 0, \quad \text{and} \quad \sigma>0.
$$
Also some additional assumptions on the set $\V$ of velocities are required. Existence results in $L^2$ were developed under these conditions in \cite{Agoshkov98,ManResSta00,EggSch10:3} by variational arguments. 
Note that the assumption $\sigma>0$ excludes the presence of void regions and that the a-priori estimates again degenerate when $\sigma \to 0$. 
Based on monotonicity arguments, existence of solutions in $L^1$ was established in \cite{Pettersson01,Falk03}, without the strict positivity assumption on $\sigma$. A similar result was obtained in \cite{BalJol08} under some restrictions on the set of velocities $\V$. The existence in $L^p$ for $1 \le p < \infty$ was actually announced in \cite{DL93vol6} but without proof. For velocities with uniform speed $|\v|$, solvability in $L^2$ was established without lower bounds on $\sigma$ in \cite{EggSch12:1}. 
While the previous results are based on some sort of contraction principle, it is possible to obtain 
existence of solutions also via compactness arguments and Riesz-Schauder or analytic Fredholm theory  \cite{Vladimirov61,StefanovUhlmann08}. These results however do not lead to computable a-priori bounds.
Let us finally also refer to  \cite{CaseZweifel63,DL93vol6,ReedSimon3} and \cite{Vladimirov61,Bardos70,Voigt85,BealsProtopopescu87,MK97,Boulanouar11} for analysis of time dependent problems, results of spectral theory, and further references. 

The aim of this manuscript is to unify and generalize previous solvability results, to relax the conditions on the parameters, and to sharpen the a-priori estimates. We will present a self-contained $L^p$ theory for stationary radiative transfer under the following assumptions:
\begin{itemize}
 \item[(A1)] Let $\V\subset \RR^3$ be open and $\R\subset \RR^{3}$ be a bounded Lipschitz domain.
 \item[(A2)] $\sigma : \R \times \V \to \RR$ is non-negative and $\sigma \tau \in L^\infty(\R \times \V)$.
 Here $\tau(\r,\v)$ denotes the length of line segment through $\r$ in direction $\v$ completely contained in $\R$; cf. e.g.  \cite{Cessenat84,Agoshkov98}.

 \item[(A3)] $k:\R\times V\times V\to \RR$ is non-negative and measurable and 
$$
\sigma - \sigma_s \ge 0 \qquad \text{and} \qquad \sigma - \sigma_s'  \ge 0.
$$
\end{itemize}
All our arguments apply almost verbatim to more general velocity spaces equipped with some positive $\sigma$-finite Radon measure $\mu$ with $\mu(\{0\})=0$; see e.g. \cite{Voigt85,BealsProtopopescu87,SanchezBourhrara2011}.
We use the assumption of an open set equipped with the Lebesgue measure mainly for ease of notation.
The first of our two main results is 
\begin{theorem}\label{thm:existence}
Let (A1)--(A3) hold and let $C_p = \frac{1}{p} \|\sigma_s \tau\|_{L^\infty} + \frac{p-1}{p} \|\sigma_s' \tau\|_{L^\infty} < \infty$. Then for all $1 \le p \le \infty$ and all admissible data $f$, $g$, the radiative transfer problem \eqref{eq:rte_eq}--\eqref{eq:rte_bc} admits a unique solution $\phi$ that satisfies
  \begin{align} \label{eq:estimate}
    \| \tau^{-\frac{1}{p}} \phi\|_{L^p(\R \times \V)} \leq e^{C_p} \big( \| \tau^{1-\frac{1}{p}} f\|_{L^p(\R \times \V)} + \|g\|_{L^p(\Gamma_-; |\n\cdot \s|)}\big).
  \end{align}
\end{theorem}
Here $\|u\|_{L^p(D;w)} = \|w^{1/p} u\|_{L^p(D)} = \big( \int_D w |u|^p \d(\r,\v)\big)^{1/p} $ is the norm of a weighted space $L^p$ space.
Almost all solvability results mentioned earlier can be obtained easily from this theorem as special cases. 
It is the use of weighted norms that allows us to derive a-priori bounds which are uniform for all $1 \le p \le \infty$. 
We will show in Section~\ref{sec:counterex} that the a-priori estimate \eqref{eq:estimate} is sharp and state estimates for the directional derivatives $\s \cdot \nabla \phi$ and the traces $\phi |_{\Gamma_+}$
in Section~\ref{sec:additional}. 

Like many of the previous results, the proof of Theorem~\ref{thm:existence} is based on a fixed-point argument. We will establish the contraction property with a factor $1-\pi$, where $\pi = 1-\exp{(-C_p)}$ can be interpreted as the probability that particles leave the domain via the outflow boundary~$\Gamma_+$. As expected, $\pi$ gets smaller when increasing the size of the domain or the scattering cross-section.
Particles may escape the system also by absorption. This case is covered by
\begin{theorem}\label{thm:apriori}
Let the conditions of Theorem~\ref{thm:existence} hold, let $\sigma>0$, and assume that for some $\nu>0$ we have $\|\sigma_s/\sigma\|_{L^\infty(\D)} \le 1-\nu$ and $\|\sigma_s'/\sigma\|_{L^\infty(\D)} \le 1-\nu$. Then 
\begin{align} \label{eq:apriori}
\|\sigma^{\frac{1}{p}} \phi\|_{L^p(\R \times \V)} 
\leq  \nu^{-1} \|\sigma^{\frac{1}{p}-1}f\|_{L^p(\R \times \V)} + \nu^{-\frac{1}{p}}\|g\|_{L^p(\Gamma_-;|\s\cdot \n|)}.
\end{align}
\end{theorem}
Now $\nu$ plays the role of a probability that particles get absorbed when interacting with the medium. 
Similar results can be found in \cite{Olhoeft62} and \cite{Vladimirov61}. 
The use of weighted norms is again essential to obtain uniform estimates for all $1 \le p \le \infty$, and it allows us to obtain also simple bounds for the directional derivatives and traces which will be stated in Section~\ref{sec:additional}.
%
In contrast to our first result, Theorem~\ref{thm:apriori} allows to consider also the case $\sigma \to \infty$ which may be important for asymptotic considerations \cite{LarsenKeller74,HabMat75,DL93vol6}.

Let us sketch the outline of the paper:
We start with reformulating \eqref{eq:rte_eq}--\eqref{eq:rte_bc} as a fixed-point problem 
and then establish the unique solvability and a-priori estimates in $L^\infty$ and $L^1$ in Sections~\ref{sec:infty} and \ref{sec:one}, respectively. 
The proof of Theorem~\ref{thm:existence} is completed in Section~\ref{sec:p} by extending these results to $L^p$, $1 \le p \le \infty$ via interpolation arguments.
In Section~\ref{sec:counterex}, we then construct a particular example showing that the estimate of Theorem~\ref{thm:existence} is sharp.
Section~\ref{sec:diffusion} contains a short proof of Theorem~\ref{thm:apriori}. 
In Section~\ref{sec:additional}, we derive estimates for derivatives and traces of the solution, and we introduce natural energy spaces for the problem \eqref{eq:rte_eq}--\eqref{eq:rte_bc}. 
Finally, we present uniform estimates for the contraction factors of the source iteration which follow easily from the contraction estimates proved earlier.
%

\section{Reformulation as fixed-point equation}

Let us start by reformulating the radiative transfer problem as an 
equivalent integral equation in the usual way \cite{CaseZweifel63}. 
We define the scattering operator by 
\begin{align} \label{eq:scattering}
  \K \phi(\r,\v) : = \int_\V k(\r,\v',\v) \phi(\r,\v')\d \v', \qquad (\r,\v) \in \D ,
\end{align}
further denote by 
\begin{align}\label{eq:extension}
 (\J g)(\r_-+t\s,\v) = e^{-\int_0^t \sigma(\r_-+s \s,\v) \d s} g(\r_-,\v), 
\qquad (\r_-,\v) \in \Gamma_-
\end{align}
the extension of boundary values, and define a lifting
\begin{align}\label{eq:explicitA}
  \L f (\r_- +t\s,\v) = \int_0^t e^{-\int_s^t \sigma(\r_-+r \s,\v) \d r} f(\r_-+s \s,\v) \d s, \qquad (\r_-,\v) \in \Gamma_-, 
\end{align}
where $0 < t < \tau(\r_-,\v)$. By elementary calculations one can verify that 
\begin{align}
(\s\cdot\nabla + \sigma) \J g &= 0,
\end{align}
and 
\begin{align}
(\s \cdot \nabla  +\sigma) \L f &= f, \qquad  \L f |_{\Gamma_-} = 0.
\end{align}
This means that the extension $\J g$ of the boundary values lies in the kernel of the differential operator and that the lifting $\L$ is a right inverse of $\s \cdot \nabla + \sigma$.
The radiative transfer problem can then be seen to be equivalent to the following operator equation in integral form \cite{CaseZweifel63}
\begin{align} \label{eq:fixedpoint}
\phi = \L \K \phi + \L f + \J g. 
\end{align}
%
To show the existence of a unique fixed-point, we will in the following sections select appropriate solution spaces, provide conditions on the data such that $\L f$ and $\J g$ lie in this space, and show that $\L \K$ is a contraction.

\section{Solvability in $L^\infty$}\label{sec:infty}

We will assume throughout that (A1)--(A3) hold and use the fact that for every point $\r \in \R$ and any velocity $\v \in \V$ we can find a point $(\r_-,\v)$ on the inflow boundary $\Gamma_-$ such that 
\begin{align} \label{eq:boundaryrep}
\r = \r_- + t \s \qquad \text{with} \quad 0 < t < \tau(\r,\v).
\end{align}
Also note that $\tau(\r,\v) = \tau(\r_-,\v)$. 
We show first that $\L \K$ is a contraction on $L^\infty(\R \times \V)$.

\begin{lemma}\label{lem:contraction_infty}
  For any $\phi\in L^\infty(\D)$ there holds
  \begin{align*}
    \| \L \K \phi\|_{L^\infty(\D)}\leq \big(1-e^{-\|\sigma_s'\tau\|_{L^\infty}}\big) \|\phi\|_{L^\infty(\D)}.
  \end{align*}
\end{lemma}
\begin{proof}
  Using $f=\K \phi$ in \eqref{eq:explicitA} and the assumption that $\sigma_s'\leq \sigma$, we obtain for $0 < t < \tau(\r_-,\v)$ 
  \begin{align*}
    |(\L \K \phi)(\r_- + t \s,\v)| 
    &\leq\int_0^t e^{-\int_s^t \sigma_s'(\r_-+r \s,\v) \d r} \sigma_s'(\r_-+s\s,\v) \d s\  \|\phi\|_{L^{\infty}(\D)} \\
    &\leq  \big(1-e^{-\|\sigma_s'\tau\|_{L^\infty}} \big)\ \|\phi\|_{L^{\infty}(\D)}.\\[-3em]
  \end{align*}
\end{proof}

Applying Banach's fixed-point theorem, we see that \eqref{eq:fixedpoint} has a unique solution $\phi \in L^\infty(\D)$ whenever $\L f$ and $\J g$ are in $L^\infty(\D)$. 
This can be guaranteed by the following two results.
\begin{lemma}\label{lem:Aboundedinv_infty}
Assume that $\tau f\in L^\infty(\D)$. Then
\begin{align*}
  \|\L f\|_{L^\infty(\D)} &\le \| \tau f\|_{L^\infty(\D)}.
\end{align*}
\end{lemma}
\begin{proof}
Using the definition of $\L$, we obtain
\begin{align*}
|\L f (\r_- +t\s,\v)| \le \int_0^{t} e^{-\int_s^t \sigma(\r_-+r \s,\v) \d r} \tau^{-1} |\tau f(\r_-+s \s,\v)| \d s \le \|\tau f\|_{L^\infty(\D)}.
\end{align*}

\vspace*{-1em}
\end{proof}
\begin{lemma}\label{lem:extension}
For any $g \in L^\infty(\Gamma_-)$ there holds
\begin{align*}
  \|\L\K\J g\|_{L^\infty(\D)} \le \| \J g\|_{L^\infty(\D)} &\leq \|g\|_{L^\infty(\Gamma_-)}.
\end{align*}
\end{lemma}
\begin{proof}
Since $\sigma \ge 0$ we immediately obtain $|\J g (\r_-+t\s,\v)| \le |g(\r_-,\v)|$, which yields the second estimate. The first one follows from Lemma~\ref{lem:contraction_infty}. 
\end{proof}
The proof reveals that $g \in L^\infty(\Gamma_-)$ is in fact necessary to ensure that $\J g$ is bounded. 
Combining the three previous Lemmas and the equivalence of the fixed-point equation \eqref{eq:fixedpoint} with the radiative transfer problem, we obtain
\begin{theorem} \label{thm:existence_Linfty}
For any $g \in L^\infty(\Gamma_-)$ and $\tau f \in L^\infty(\D)$, problem \eqref{eq:rte_eq}--\eqref{eq:rte_bc} has a unique solution $\phi \in L^\infty(\D)$ which satisfies the a-priori bounds 
$$ 
\|\phi\|_{L^\infty(\D)} \le e^{\|\sigma_s' \tau\|_{L^\infty}} \big( \|\tau f\|_{L^\infty(\D)} + \|g\|_{L^\infty(\Gamma_-)} \big).
$$
\end{theorem}
\begin{proof}
The existence of a unique fixed-point for \eqref{eq:fixedpoint} follows from Lemma~\ref{lem:contraction_infty} and Banach's fixed point theorem. By the previous estimates, we get 
$$ 
\|\phi\|_{L^\infty(\D)} \le (1-e^{-\|\sigma_s' \tau\|_{L^\infty}}) \|\phi\|_{L^\infty(\D)} + \|\tau f\|_{L^\infty(\D)} + \|g\|_{L^\infty(\Gamma_-)},
$$
from which the assertion is derived straight forward.
\end{proof}
This completes the proof of Theorem~\ref{thm:existence} for the case $p=\infty$. 
Note that actually no condition on the cross-section $\sigma_s$ was required here.

\section{Solvability in $L^1$}\label{sec:one}

Setting $w = \s \cdot \nabla \phi + \sigma \phi$ allows us to express the solution as $\phi = \J g + \L w$. 
The fixed-point problem \eqref{eq:fixedpoint} can then be stated equivalently as
\begin{align} \label{eq:fixedpoint2}
w =  \K \L w + f + \K \J g, \qquad \phi = \L w + \J g.
\end{align}
We want to show existence of a unique fixed-point for \eqref{eq:fixedpoint2} in $L^1(\D)$. To do so, we will first establish the contraction property for the operator $\K \L$. We will make use of the following well-known integral formula 
\begin{align}\label{eq:intformula}
 \int_\D f(\r,\v) \d(\r,\v) 
&= \int_{\Gamma_-}\int_0^{\tau(\r_-, \v)} f(\r_-+t \s,\v) |\s\cdot \n|\d t \d(\r_-,\v) 
\end{align}
which directly follows from Fubini's theorem; see e.g. \cite{Vladimirov61,Agoshkov98,ChoulliStefanov98}. We can then show 
%
%
\begin{lemma} \label{lem:contraction_l1}
For any $w \in L^1(\D)$ there holds
$$ 
  \|\K \L w\|_{L^1(\D)} \le (1-e^{-\|\sigma_s \tau\|_{L^\infty}})\|w\|_{L^1(\D)}.
$$ 
\end{lemma}
\begin{proof}
By the definitions of $\K$ and $\sigma_s$, we get
\begin{align*}
\|\K \L w\|_{L^1(\D)} 
  &\leq \int_\R \int_\V \int_\V k(\r,\v',\v) |(\L w) (\r,\v')| \d\v' \d\v \d\r \\
  &= \int_\R \int_V \sigma_s(\r,\v') |(\L w)(\r,\v')| \d\v' \d\r = (*).
\end{align*}
Using the definition of $\L$ and applying the integral formula \eqref{eq:intformula} further yields
\begin{align*}
(*) 
&\leq \int_{\Gamma_-} \! \int_0^{\tau(\r_-,\v)}
 \sigma_s(\r_-+t\s,\v) \int_0^t e^{-\int_s^t \sigma(r_-+r\s,\v) \d r} |w(\r_-+s\s,\v)| \d s \d t |\n \cdot \v| \d(\r_-,\v) \\
&\leq \int_{\Gamma_-} \int_0^{\tau(\r_-,\v)} \big(1-e^{-\int_s^{\tau(\r_-,\v)} \sigma_s(\r_-+r\s,\v) \d r} \big) |w(\r_-+s\s,\v)| \d s |\n \cdot \v| \d(\r_-,\v)
.
\end{align*}
Here we used $\sigma_s\leq \sigma$ and applied Fubini's theorem again to exchange the order of integrals with respect to $\d s$ and $\d t$ and explicitly computed the latter. The assertion now follows directly. 
\end{proof}
A slightly weaker result was proven in a similar manner in \cite{BalJol08}. 
To establish the existence of a fixed-point, we additionally have to require that $f$ and $\K \J g$ are in $L^1(\D)$. For the latter term,  we use
\begin{lemma} \label{lem:Kextension}
For any $ g \in L^1(\Gamma_-;|\s \cdot \n|)$ there holds
$$
\|\K \J g\|_{L^1(\D)} \le \|\sigma_s \J g\|_{L^1(\D)} \le (1-e^{-\|\sigma_s \tau\|_{L^\infty}})\|g\|_{L^1(\Gamma_-;|\s \cdot \n|)}.
$$ 
\end{lemma}
\begin{proof} 
By the definition of $\K$ and $\sigma_s$, we obtain
\begin{align*}
 \| \K \J g\|_{L^1(\D)} \leq  \int_{\R} \int_\V \sigma_s(\r,\v') |\J g(\r,\v')|\d \v'\d \r = \|\sigma_s \J g\|_{L^1(\D)}.
\end{align*}
Employing the definition of $\J$ and the integral formula \eqref{eq:intformula}, yields
\begin{align*}
& \| \sigma_s \J g\|_{L^1(\D)} 
\le (1-e^{-\|\sigma_s \tau\|_{L^\infty}}) \|g\|_{L^1(\Gamma_-; |\s\cdot \n|)},
  \end{align*}
where in the last step, we used $\sigma_s \le \sigma$ and a direct computation of the integral similar as in the proof of Lemma~\ref{lem:Kextension}.
\end{proof}
By Banach's fixed-point theorem and the previous estimates, we now obtain
\begin{lemma} \label{lem:existence_w}
For any $f \in L^1(\D)$ and $g \in L^1(\Gamma_-;|\s \cdot \n|)$, 
the fixed-point problem \eqref{eq:fixedpoint2} has a unique solution $w \in L^1(\D)$ and there holds
\begin{align*} 
\|w\|_{L^1(\D)} \le e^{\|\sigma_s \tau\|_{L^\infty}} \big( \|f\|_{L^1(\D)} + (1-e^{-\|\sigma_s \tau\|_{L^\infty}})\|g\|_{L^1(\Gamma_-;|\s \cdot \n|)}\big).
\end{align*}
\end{lemma}
To establish an $L^1$ estimate for the solution $\phi = \L w + \J g$ of 
problem \eqref{eq:rte_eq}--\eqref{eq:rte_bc}, we have to establish additional bounds for $\L w$ and $\J g$.
\begin{lemma} \label{lem:l1est}
For any $w \in L^1(\D)$ and any $g \in L^1(\Gamma_-;|\s \cdot \n|)$ there holds
$$
\|\tau^{-1} \L w\|_{L^1(\D)} \le \|w\|_{L^1(\D)} 
\qquad \text{and} \qquad 
  \|\tau^{-1} \J g\|_{L^1(\D)} \le \|g\|_{L^1(\Gamma_-;|\s \cdot \n|)}.
$$
\end{lemma}
\begin{proof}
These estimates follow from the integral formula \eqref{eq:intformula} and direct computations.
\end{proof}

A combination of the previous estimates now yields 
\begin{theorem}\label{thm:existence_L1}
For any $f \in L^1(\D)$ and $g\in L^1(\Gamma_-; |\s\cdot \n|)$, the boundary value problem \eqref{eq:rte_eq}--\eqref{eq:rte_bc} has a unique solution $\phi \in L^1(\D)$ which satisfies
\begin{align*}
  \|\tau^{-1} \phi\|_{L^1(\D)} \le e^{\|\sigma_s \tau\|_{L^\infty}} \big( \|f\|_{L^1(\D)} + \|g\|_{L^1(\Gamma_-;|\s \cdot \n|)} \big) .
\end{align*}
\end{theorem}
\begin{proof}
The result follows from the representation $\phi = \L w + \J g$ of the solution by applying the triangle inequality and using the estimates of Lemmas~\ref{lem:existence_w} and \ref{lem:l1est}.
\end{proof}

This completes the proof Theorem~\ref{thm:existence} for the case $p=1$. Note that for our arguments, we did not use the condition on the scattering cross-section $\sigma_s'$ here.

\section{Solvability in $L^p$ and proof of Theorem~\ref{thm:existence}} \label{sec:p}

For establishing solvability in $L^p$, we will utilize the results for $L^1$ and $L^\infty$ and the complex method of interpolation \cite{BerghLoefstroem,Lunardi2009}. Let us recall that for a $\sigma$-finite measure space $(\M, \d\mu)$ 
$$
L^p(\M;\d\mu) = [L^1(\M;\d\mu),L^\infty(\M; \d\mu)]_\theta, \qquad \theta=\frac{p-1}{p}, 
$$
i.e., for $1 \le p \le \infty$ the space $L^p$ is an interpolation space between $L^1$ and $L^\infty$. 
In addition, the interpolation norm coincides with the norm of $L^p$; see \cite[Example~2.1.11]{Lunardi2009}.

We are now in the position to complete the 
\begin{proof}[{\bf Proof of Theorem~\ref{thm:existence}}]
As a first step, let us establish the a-priori estimate for data that simultaneously satisfy the requirements of Theorems~\ref{thm:existence_Linfty} and \ref{thm:existence_L1}. Noting that 
$$
\|\phi\|_{L^p(\D;\tau^{-1})}=\|\tau^{-\frac{1}{p}} \phi\|_{L^p(\D)},
$$ 
the a-priori bounds of these previous results can be written as
$$
\|\phi\|_{L^p(\D; \tau^{-1})} \le e^{C_p} \big( \| \tau f \|_{L^p(\D;\tau^{-1})} + \|g\|_{L^p(\Gamma_-;|\s \cdot \n|)}\big) \quad \text{for }  p \in \{1,\infty\}.
$$  
Here $C_p$ denotes the stability constant from Theorem~\ref{thm:existence}. 
Using the linearity of the problem, we can decompose $\phi = \phi_g + \phi_f$, where $\phi_g$ and $\phi_f$ are the solutions of \eqref{eq:rte_eq}--\eqref{eq:rte_bc} with $f \equiv 0$ and $g \equiv 0$, respectively.
An application of the Riesz-Thorin theorem \cite{BerghLoefstroem,Lunardi2009} then yields 
$$
\|\phi_f\|_{L^p(\D; \tau^{-1})} \le e^{C_p}  \| \tau f \|_{L^p(\D;\tau^{-1})}
\quad \text{and} \quad \|\phi_g\|_{L^p(\D; \tau^{-1})} \le e^{C_p} \|g\|_{L^p(\Gamma_-;|\s \cdot \n|)}
$$
for any $1 \le p \le \infty$. From this estimate the a-priori estimate is derived via the triangle inequality.
The unique solvability for all admissible data follows by a density argument.
\end{proof}

\section{Sharpness of the a-priori estimates}  \label{sec:counterex}

In the following we show by example that the a-priori bound of Theorem~\ref{thm:existence} is sharp. 
Let $\R=\mathcal{B}_1(0)$ be the unit ball in $\RR^3$ and let $\V =\S^2$ be the unit sphere; note that hence $\s = \v$ in the following. 
We consider the scattering 
operator
\begin{align*}
  \K \phi(\r,\v) 
=  \sigma(\r,\v) \phi(\r,-\v).
\end{align*}
This definition yields the essential property
$$
|\K \phi(\r,\v)| \leq \sigma(\r,\v) \|\phi\|_{L^\infty(\D)}
$$ 
An inspection of the previous results shows, that Lemma~\ref{lem:contraction_infty} and therefore all result of Section~\ref{sec:infty} hold true with $\sigma_s'=\sigma$ also for this example.
Equation \eqref{eq:rte_eq} can then be written as
\begin{align*}
 \v \cdot \nabla \phi(\r,\v) + \sigma(\r,\v) \big( \phi(\r,\v) -  \phi(\r,-\v)\big) &= f(\r,\v).
\end{align*}
This construction yields that only directions $\v$ and $-\v$ are coupled in the transport equation. 
Let us fix one direction $\v$ and write $\phi^\pm(t)= \phi(\r_- +t \tau(\r_-,\v) \v,\pm \v)$. We then obtain
\begin{align*}
 \frac{1}{\tau} \frac{d}{dt} \phi^+ + \sigma^+ \big( \phi^+ -  \phi^-\big) &= f^+\quad \text{in } (0,1) \quad \text{with } \phi^+(0)=0, \qquad \text{and}\\
 -\frac{1}{\tau} \frac{d}{dt} \phi^- + \sigma^- \big( \phi^- - \phi^+\big) &= f^-\quad \text{in } (0,1) \quad \text{with }\phi^-(1)=0,
\end{align*}
and the solutions of these equations are given by
\begin{align*}
 \phi^+(t) = \phi^-(1-t) = 1- e^{-\int_0^t \sigma^+(s) \d s}.
\end{align*}
For $l\geq 3$ and $k=2^{l+3}$, let us choose 
\begin{align*}
 \sigma^+(t) &= \sigma^-(1-t)= k^l (1-t)^k,\qquad \text{and} \\
 f^+(t) &= f^-(1-t) = \sigma^+(t) e^{-\int_0^{1-t} \sigma^+(s)\d s}.
\end{align*}
After some basic calculations, one can see that 
\begin{align*}
 \|\phi^\pm\|_{L^\infty(0,1)} &= \phi^+(1) =\phi^-(0)= 1-\exp\big(-\frac{k^l}{k+1}\big) =: a(l),
\end{align*}
and that 
\begin{align*}
  \|\tau f^\pm\|_{L^\infty(\D)} 
  &= f^+(0) = f^-(1) \leq \exp\big(-(1-\frac{1}{l})2^{l(l+3)-(l+4)}\big) \\
  &= \exp\Big(-(1-\frac{1}{l})\|\sigma^\pm\tau\|_{\infty}^{1-\frac{1}{l}\frac{l+4}{l+3}}\Big) =: b(l) e^{-\|\sigma^\pm \tau \|}.
\end{align*}
Note that $a(l)$ and $b(l)$ tend to one as $l$ goes to infinity. Combining these estimates yields
$$
\|\phi^\pm\|_{L^\infty} = a(l) \ge \frac{a(l)}{b(l)} e^{\|\sigma^\pm\tau\|_{L^\infty}} \|\ell f^\pm\|_{L^\infty}.  
$$
This construction can be repeated for all directions $\v$. 
Since we had $\sigma_s'=\sigma$ here, this shows that the estimate of Theorem~\ref{thm:existence} is sharp at least in the case $p=\infty$.


\section{Proof of Theorem~\ref{thm:apriori}}\label{sec:diffusion}

We now illustrate that better a-priori estimates can be obtained, if some absorption is present. 
Let the assumptions of Theorem~\ref{thm:apriori} hold. Formally multiplying \eqref{eq:rte_eq} with $ \phi |\phi|^{p-2}$ yields
\begin{align*}
 \frac{1}{p} \s\cdot \nabla |\phi|^p + \sigma |\phi|^p = \phi |\phi|^{p-2} \K \phi + f \phi |\phi|^{p-2}.
\end{align*}
By integrating this equation over $\D$, performing integration by parts, using the boundary conditions \eqref{eq:rte_bc}, and rearranging terms, we obtain
\begin{align*}
&\|\sigma^{\frac{1}{p}}\phi\|_{L^p(\D)}^p + \frac{1}{p} \|\phi\|_{L^p(\Gamma_+;|\s\cdot \n|)}^p \\
& \qquad = \int_\D \phi |\phi|^{p-2} \K \phi\d(\r,\v) + \int_\D f \phi |\phi|^{p-2}\d(\r,\v) +  \frac{1}{p} \|g\|_{L^p(\Gamma_-;|\s\cdot \n|)}^p.
\end{align*}
The first term on the right hand side can be estimated with H\"older's inequality by
\begin{align*}
\int_\D \phi |\phi|^{p-2} \K \phi\d(\r,\v) 
  &\leq \|\sigma^{\frac{1}{p}}\phi\|_{L^p(\D)}^{p-1} \|\sigma^{\frac{1-p}{p}}\K \phi\|_{L^p(\D)} 
  \leq (1-\nu) \|\sigma^{\frac{1}{p}} \phi\|_{L^p(\D)}^p.
\end{align*}
For the last step, we used the following basic estimates for the scattering operator
\begin{align*}
 \| \K \phi\|_{L^1}\leq \|\frac{\sigma_s}{\sigma}\|_{L^\infty} \|\sigma \phi\|_{L^1} 
\quad \text{ and } \quad 
\|\frac{1}{\sigma}\K \phi\|_{L^\infty}\leq \|\frac{\sigma_s'}{\sigma}\|_{L^\infty} \|\phi\|_{L^\infty}
\end{align*}
from which one obtains by interpolation that $\|\sigma^{\frac{1-p}{p}}\K \phi\|_{L^p} \le \|\sigma^{\frac{1}{p}} \phi\|_{L^p}$.
To bound the term involving the right hand side $f$, we apply H\"older's and Young's inequality, to get
\begin{align*}
 \int_\D f \phi |\phi|^{p-2}\d(\r,\v)  
  &\leq \|\sigma^{\frac{1-p}{p}}f \|_{L^p(\D)} \|\sigma^{\frac{1}{p}}\phi\|_{L^p(\D)}^{p-1} \\
  &\leq \nu^{1-p}  \frac{1}{p} \|\sigma^{\frac{1-p}{p}}f \|_{L^p(\D)}^p + \nu \frac{p-1}{p} \|\sigma^{\frac{1}{p}}\phi\|_{L^p(\D)}^{p}.
\end{align*}
Putting all estimates together and multiplying by $p$, finally leads to
\begin{align*} 
\nu \|\sigma^{\frac{1}{p}} \phi\|_{L^p(\D)}^p  + \|\phi\|_{L^p(\Gamma_+;|\s\cdot\n|)}^p \leq  \nu^{1-p}\|\sigma^{\frac{1-p}{p}} f \|_{L^p(\D)}^p + \|g\|_{L^p(\Gamma_-;|\s\cdot\n|)}^p.
\end{align*}
From this estimate, the assertion of Theorem~\ref{thm:apriori} now follows directly.



\section{Additional results}\label{sec:additional}

To complete our discussion, we collect in the following some further results which follow more or less directly from our previous considerations. 

\subsection{Estimates for the derivatives}

Using the a-priori estimates of Theorem~\ref{thm:existence} and the fixed-point equation \eqref{eq:fixedpoint}, 
it is straight-forward to obtain also estimates for the directional derivatives $\s \cdot \nabla \phi$.
Let us first consider the case $p=1$, where we have
\begin{lemma} \label{lem:derivative_L1}
Under the assumptions of Theorem~\ref{thm:existence_L1} one has
$$ 
\|\s \cdot \nabla \phi\|_{L^1(\D)} \le 2 e^{\|\sigma_s \tau\|_{L^\infty}} \big( \|f\|_{L^1(\D)} + \|g\|_{L^1(\Gamma_-;|\s \cdot \n|)}\big). 
$$ 
\end{lemma}
\begin{proof}
By $\phi = \L w + \J g$ and the properties of the operators $\L$ and $\J$, we obtain
$$
\|\s \cdot \nabla \phi\|_{L^1(\D)} \le \|w\|_{L^1(\D)} + \|\sigma \phi\|_{L^1(\D)}.
$$
The first term can be estimated by Lemma~\ref{lem:existence_w}, and for the second, we use
$$
\|\sigma \phi\|_{L^1(\D)} \le \|\sigma\L w\|_{L^1(\D)}+ \|\sigma \J g\|_{L^1(\D)}\leq \| w\|_{L^1(\D)}+ \|g\|_{L^1(\Gamma_-;|\s\cdot \n|)}.
$$
The second estimate for the boundary term is obtained as in Lemma~\ref{lem:Kextension}.
\end{proof}

For the case $p=\infty$, we have
\begin{lemma} \label{lem:derivative_Linfty}
Under the assumptions of Theorem~\ref{thm:existence_Linfty} there holds
$$ 
\|\tau \s \cdot \nabla \phi\|_{L^\infty(\D)} \le \big(1+2\|\sigma\tau\|_{L^\infty} \big) e^{\|\sigma_s' \tau\|_{L^\infty(\D)}} \big( \|f\|_{L^\infty(\D)} + \|g\|_{L^\infty(\Gamma_-;|\s \cdot \n|)}) .
$$
\end{lemma}
\begin{proof}
The identity $\s\cdot\nabla \phi = \K \phi - \sigma\phi  + f$ yields
$$
\|\tau \s \cdot \nabla \phi\|_{L^\infty(\D)} \le (\|\tau \sigma_s'\|_{L^\infty(\D)} + \|\tau \sigma\|_{L^\infty(\D)}) \|\phi\|_{L^\infty(\D)} + \|\tau f\|_{L^\infty(\D)}.
$$
The estimate then follows from the bounds of Theorem~\ref{thm:existence_Linfty} and the condition $\sigma_s' \le \sigma$.
\end{proof}


Arguing as in the proof of Theorem~\ref{thm:existence}, the case $1 \le p \le \infty$ is then covered by
\begin{theorem} \label{thm:derivatives_Lp}
Under the assumptions of Theorem~\ref{thm:existence} there holds
\begin{align*}
 \| \tau^{1-\frac{1}{p}}\s\cdot\nabla \phi\|_{L^p(\D)}\leq 2(1+\|\sigma \tau\|_{L^\infty}) e^{C_p} \big( \|\tau^{1-\frac{1}{p}} f\|_{L^p(\D)} + \|g\|_{L^p(\Gamma_-;|\s \cdot \n|)} \big).
\end{align*}
\end{theorem}
 
Again, we can obtain in a similar way stronger estimates under additional assumptions.
\begin{theorem} \label{thm:derivatives_Lp_strong}
Let the conditions of Theorem~\ref{thm:apriori} hold. Then
\begin{align*}
 \| \sigma^{\frac{1}{p}-1}\s\cdot\nabla \phi\|_{L^p(\D)}\leq  2 \nu^{-1} \|\sigma^{\frac{1}{p}-1} f\|_{L^p(\D)} + 2 \nu^{-1/p}\|g\|_{L^p(\Gamma_-;|\s \cdot \n|)}.
\end{align*}
\end{theorem}

The use of weighted norms again substantially simplifies the derivation of these results.

\subsection{Energy space and a trace lemma}

The norms in which we obtained the a-priori estimates of Theorem~\ref{thm:existence} and Theorem~\ref{thm:derivatives_Lp} suggest to define the following energy space
\begin{align} \label{eq:energyspace}
\W^p=\{ \phi: \D \to \RR  :  \tau^{-\frac{1}{p}} \phi \in L^p(\D),\ \tau^{1-\frac{1}{p}} \s\cdot\nabla \phi \in L^p(\D)\}.
\end{align}
The natural norm for this space is given by 
$$
\|\phi\|_{\W^p}^p = \|\tau^{-\frac{1}{p}} \phi\|_{L^p(\D)}^p + \|\tau^{1-\frac{1}{p}} \s \cdot \nabla \phi\|_{L^p(\D)}^p.
$$ 
For functions in the space $\W^p$, we have the following result for traces.
\begin{theorem} \label{thm:trace}
The trace operators $\gamma_\pm:\W^p\to L^p(\Gamma_\pm;|\s\cdot\n|)$ are continuous and surjective.
\end{theorem}
A proof of this statement follows easily by direct computation; see also \cite{Boulanouar09}. 
For other estimates and general material about traces for radiative transfer problems let us refer to \cite{Cessenat84,ManResSta00,ChoulliStefanov98}.

Using the results of Sections~\ref{sec:infty}--\ref{sec:p}, we also obtain that
$$
\|\tau^{1-\frac{1}{p}} \big( \s \cdot \nabla \phi + \sigma \phi -\K\phi \big) \|_{L^p(\D)} \le  (1 + 2\|\sigma \tau\|_{L^\infty}) \|\phi\|_{\W^p}.
$$
The individual operators could be estimated in the same way. 
Summarizing, we obtain 
\begin{theorem} \label{thm:isomorphism}
Let (A1)--(A3) hold. Then the mapping
$$
\W^p \to  L^p(\D;\tau^{p-1}) \times L^p(\Gamma_-;|\s \cdot \n|),
\quad  \phi \mapsto (\s\cdot\nabla\phi +\sigma \phi -\K\phi,\ \gamma_- \phi)
$$
is continuous and boundedly invertible. 
\end{theorem}
This result shows that the assumptions on the data cannot be relaxed when searching for solutions in the energy space $\W^p$. 
Under the stronger assumptions of Theorem~\ref{thm:apriori}, we can define in a similar manner an energy space
$$
\widetilde \W^p = \{\sigma^{\frac{1}{p}} \phi \in L^p(\D),\ \sigma^{\frac{1}{p}-1} \s\cdot\nabla \phi \in L^p(\D)\}.
$$ 
Results analogous to Theorem~\ref{thm:trace} and \ref{thm:isomorphism} can easily be derived also for this space. 
For the corresponding statements it suffices to replace the weight function $\tau$ by $\sigma^{-1}$; compare also with Theorem~\ref{thm:existence} and \ref{thm:apriori}.

\subsection{Spectral estimates and convergence of the fixed-point iterations}

The solvability results of the previous sections were based on Banach's fixed-point theorem. 
The corresponding fixed-point iteration reads 
\begin{align} \label{eq:sourceiteration}
\phi_{n+1} = \L \K \phi_n + \L f + \J g.
\end{align}
We show now that under our general assumptions (A1)--(A3), 
the spectral radius of the fixed-point operator $\L \K$ is always uniformly bounded away from one.
\begin{theorem} \label{thm:spectral}
Let (A1)--(A3) hold. Then  for all $1 \le p \le \infty$ 
$$
\rho_p(\L \K) := \lim_{n \to \infty} \sqrt[n]{\|(\L \K)^n\|_{L^p(\D;\tau^{-1})}} \le 1-e^{-C_p}.
$$
\end{theorem}
\begin{proof}
The case $p=\infty$ follows immediately from Lemma~\ref{lem:contraction_infty}. 
For $p=1$, on the other hand, we can estimate the powers of the fixed-point operator by
\begin{align*}
\|(\L \K)^n\|_{L^1(\D;\tau^{-1})} =  \|\tau^{-1}(\L \K)^n \tau\|_{L^1(\D)} 
    \le \|\tau^{-1}\L\|_{L^1(\D)} \|\K \L\|_{L^1(\D)}^{n-1} \|\K \tau \|_{L^1(\D)}.  
\end{align*}
The first two terms can be bounded by Lemma~\ref{lem:Kextension} and \ref{lem:contraction_l1}, and for the third term we use the estimate $\|\K \tau\|_{L^1(\D)} \le \|\sigma_s \tau\|_{L^\infty(\D)}$. From this we obtain the estimate for the spectral radius for $p=1$. The general case then follows again by interpolation arguments.
\end{proof}
Our analysis thus shows that under the weak sub-criticality assumptions (A3), the source iteration \eqref{eq:sourceiteration} converges in $L^p$ for any $1 \le p \le \infty$ with a contraction factor $1-e^{-C_p}$. 
Note that no positive lower bounds on the absorption are needed for the convergence. The same arguments may  be used to analyze other fixed-point iterations, cf.\@ \cite{Abramov08}.




\begin{thebibliography}{10}

\bibitem{Abramov08}
B.~D. Abramov.
\newblock Methods of iterations on subdomains for neutron transport theory
  problems.
\newblock {\em Transport Theory and Statistical Physics}, 37(2--4):208--235,
  2008.

\bibitem{Agoshkov98}
V.~Agoshkov.
\newblock {\em Boundary Value Problems for Transport Equations}.
\newblock Modeling and Simulation in Science, Engineering and Technology.
  Birkh\"auser, Boston, 1998.

\bibitem{Arridge99}
S.~R. Arridge.
\newblock Optical tomography in medical imaging.
\newblock {\em Inverse Problems}, 15(2):R41--R93, 1999.

\bibitem{BalJol08}
G.~Bal and A.~Jollivet.
\newblock Stability estimates in stationary inverse transport.
\newblock {\em Inverse Probl. Imaging}, 2(4):427--454, 2008.

\bibitem{Bardos70}
C.~Bardos.
\newblock Probl\`emes aux limites pour les \'equations aux d\'eriv\'ees
  partielles du premier ordre \`a coefficients r\'eels; th\'eor\`emes
  d'approximation; application \`a l'\'equation de transport.
\newblock {\em Annales scientifiques de l'\'E.N.S.}, 3(2):185--233, 1970.

\bibitem{BealsProtopopescu87}
R.~Beals and V.~Protopopescu.
\newblock Abstract time-dependent transport equations.
\newblock {\em Journal of Mathematical Analysis and Applications},
  121:370--405, 1987.

\bibitem{BerghLoefstroem}
J.~Bergh and J.~L\"ofstr\"om.
\newblock {\em Interpolation Spaces -- An Introduction}.
\newblock Springer, Berlin, 1976.

\bibitem{Boulanouar09}
M.~Boulanouar.
\newblock New trace theorems for neutronic function spaces.
\newblock {\em Transport Theory and Statistical Physics}, 38(4):228--242, 2009.

\bibitem{Boulanouar11}
M.~Boulanouar.
\newblock New results in abstract time-dependent transport equations.
\newblock {\em Transport Theory and Statistical Physics}, 40(2):85--125, 2011.

\bibitem{CaseZweifel63}
K.~M. Case and P.~F. Zweifel.
\newblock Existence and uniqueness theorems for the neutron transport equation.
\newblock {\em Journal of Mathematical Physics}, 4(11):1376--1385, 1963.

\bibitem{CaseZweifel67}
K.~M. Case and P.~F. Zweifel.
\newblock {\em Linear transport theory}.
\newblock Addison-Wesley Publishing Co., Reading, 1967.

\bibitem{Cercignani:1988}
C.~Cercignani.
\newblock {\em The {B}oltzmann Equation and Its Applications}.
\newblock Springer-Verlag, Berlin, 1988.

\bibitem{Cessenat84}
M.~Cessenat.
\newblock Th\'eor\`emes de trace {$L^p$} pour des espaces de fonctions de la
  neutronique.
\newblock {\em C. R. Acad. Sci. Paris S\'er. I Math.}, 299:831--834, 1984.

\bibitem{ChoulliStefanov98}
M.~Choulli and P.~Stefanov.
\newblock An inverse boundary value problem for the stationary transport
  equation.
\newblock {\em Osaka J. Math.}, 36(1):87--104, 1998.

\bibitem{DL93vol6}
R.~Dautray and J.~L. Lions.
\newblock {\em Mathematical Analysis and Numerical Methods for Science and
  Technology, Evolution Problems {II}}, volume~6.
\newblock Springer, Berlin, 1993.

\bibitem{DM79}
J.~J. Duderstadt and W.~R. Martin.
\newblock {\em Transport Theory}.
\newblock John Wiley \& Sons, Inc., New York, 1979.

\bibitem{EggSch10:3}
H.~Egger and M.~Schlottbom.
\newblock A mixed variational framework for the radiative transfer equation.
\newblock {\em M3AS}, 03(22), 2012.

\bibitem{EggSch12:1}
H.~Egger and M.~Schlottbom.
\newblock On unique solvability for stationary radiative transfer with
  vanishing absorption.
\newblock 2012.

\bibitem{Falk03}
L.~Falk.
\newblock Existence of solutions to the stationary linear boltzmann equation.
\newblock {\em Transport Theory and Statistical Physics}, 32(1):37--62, 2003.

\bibitem{HabMat75}
G.~J. Habetler and B.~J. Matkowsky.
\newblock Uniform asymptotic expansions in transport theory with small mean
  free paths, and the diffusion approximation.
\newblock {\em Journal of Mathematical Physics}, 16(4):846--854, 1975.

\bibitem{LarsenKeller74}
E.~W. Larsen and J.~B. Keller.
\newblock Asymptotic solution of neutron transport problems for small mean free
  paths.
\newblock {\em J. Math. Phys.}, 15(1):75--81, 1974.

\bibitem{Lunardi2009}
A.~Lunardi.
\newblock {\em Interpolation Theory}.
\newblock Edizione della Normale, Pisa, 2009.

\bibitem{ManResSta00}
T.~A. Manteuffel, K.~J. Ressel, and G.~Starke.
\newblock A boundary functional for the least-squares finite-element solution
  for neutron transport problems.
\newblock {\em SIAM J. Numer. Anal.}, 2:556--586, 2000.

\bibitem{MK97}
M.~Mokthar-Kharroubi.
\newblock {\em Mathematical Topics in Neutron Transport Theory}.
\newblock World Scientific, Singapore, 1997.

\bibitem{Olhoeft62}
J.~E. Olhoeft.
\newblock {\em The Doppler Effect for Non-Uniform Temperatures}.
\newblock PhD thesis, University of Michigan, 1962.

\bibitem{Peraiah04}
A.~Peraiah.
\newblock {\em An Introduction to Radiative Transfer -- Methods and
  applications in astrophysics}.
\newblock Cambridge University Press, 2004.

\bibitem{Pettersson01}
R.~Pettersson.
\newblock On stationary solutions to the linear boltzmann equation.
\newblock {\em Transport Theory and Statistical Physics}, 30(4--6):549--560,
  2001.

\bibitem{Pomraning73}
G.~C. Pomraning.
\newblock {\em The Equations of Radiation Hydrodynamics}.
\newblock Pergamon Press, Oxford, 1973.

\bibitem{ReedSimon3}
M.~Reed and B.~Simon.
\newblock {\em Methods of modern mathematical physics {III}: Scattering
  Theory}.
\newblock Academic Press, San Diego, 1979.

\bibitem{SanchezBourhrara2011}
R.~Sanchez and L.~Bourhrara.
\newblock Existence result for the kinetic neutron transport problem with a
  general albedo boundary condition.
\newblock {\em Transport Theory and Statistical Physics}, 40(2):69--84, 2011.

\bibitem{StefanovUhlmann08}
P.~Stefanov and G.~Uhlmann.
\newblock An inverse source problem in optical molecular imaging.
\newblock {\em Anal. PDE}, 1:115--126, 2008.

\bibitem{Vladimirov61}
V.~S. Vladimirov.
\newblock Mathematical problems in the one-velocity theory of particle
  transport.
\newblock Technical report, Atomic Energy of Canada Ltd. AECL-1661. translated
  from Transactions of the V.A. Steklov Mathematical Institute (61), 1961.

\bibitem{Voigt85}
J.~Voigt.
\newblock Spectral properties of the neutron transport equation.
\newblock {\em Journal of Mathematical Analysis and Applications},
  106:140--153, 1985.

\end{thebibliography}

\end{document}